\documentclass[11pt]{amsart}
\usepackage{amssymb,latexsym,epsfig,color}

\newcommand{\Z}{{\mathbb Z}}
\newcommand{\N}{{\mathbb N}}
\newcommand{\R}{{\mathbb R}}

\newcommand{\Set}{{\mathcal S}}

\newtheorem{theorem}{Theorem}[section]
\newtheorem{proposition}[theorem]{Proposition}
\newtheorem{corollary}[theorem]{Corollary}
\newtheorem{lemma}[theorem]{Lemma}

\newtheorem{question}[theorem]{Question}

\theoremstyle{definition}

\newtheorem{remark}[theorem]{Remark}
\newtheorem{example}[theorem]{Example}

\title [on the occurrence of large gaps in small contingency tables]
{on the occurrence of large gaps \\ in small contingency tables}
\author{Edwin O'Shea}
\thanks{This research was made possible by visits to the 
Statistical and Applied Mathematical Sciences Institute (SAMSI) 
as part of their thematic program on ``Algebraic Methods in Systems Biology and Statistics''. 
The author is supported by The De Br{\' u}n Centre for Computational Algebra at NUI Galway which 
is funded under Science Foundation Ireland's Mathematics Initiative.}
\address{De Br{\' u}n Centre for Computational Algebra, School of Mathematics, 
NUI Galway, 1 University Road, Galway, IRELAND}
\email{edwin.oshea@nuigalway.ie}
\date{\today}

\begin{document}

\begin{abstract} 
Examples of small contingency tables on binary random variables with large integer 
programming gaps on the lower bounds of cell entries were constructed by Sullivant. 
We argue here that the margins for which these constructed large gaps occur are 
rarely encountered, thus reopening the question of {\em whether linear 
programming is an effective heuristic for detecting disclosures when releasing margins of 
multi-way tables}. The notion of {\em rarely encountered} is made precise through 
the language of standard pairs. 
\end{abstract}

\maketitle

\section{Introduction}

When governmental bodies like census bureaus collect data there is a tension 
between publicly releasing as much information as possible while at the same time 
striving to ensure that any one individual's private data cannot be discerned 
from the released information. One such case of this is when each piece of private 
data is recorded in a cell in a multi-dimensional contingency table and the released data 
is a collection of smaller margin tables (or simply {\em margins}) which portray the 
interactions between some subsets of the variables. These margins are simply 
higher dimensional analogues of row sums and column sums on the contingency table. So how 
would we know that a person's individual data is protected despite the release of possibly 
many margins~?

There are many criteria that may have to be accounted for in limiting the disclosure of private data: 
see for example \cite{FieSla} on issues with small cell counts in sparse contingency tables, 
or see \cite{PLOS} for privacy concerns in genetic databases (see \cite[Page 5]{AAAS} for an 
introductory discussion). 
One such measure is the practical difficulty in attaining bounds for cell 
entries that can be discerned from the released margins. 
One way in which these bounds can be found is from standard methods in integer programming 
but solving the general integer program has a theoretical complexity of NP-complete 
\cite[\S 18.1]{SchBlue} and, 
practically speaking, are very challenging to solve. On the other hand, linear programs can 
be solved in polynomial time and there is significant high-powered software designed 
specifically to practically solve linear programs. So one must know if the linear 
relaxations of the integer programs associated to disclosure limitation are {\em always} good 
approximations for bounding cells in contingency tables. 

Based on theoretical results for $2$-way tables and practical experience on higher dimensional 
contingency tables, it was thought (\cite{CDKRM}, \cite{DobFie}) that the linear 
approximation was always reliable for bounding cells. 
However, using Gr{\" o}bner basis techniques, Sullivant \cite[Theorem 1]{SulLargeGaps} 
constructed a family of contingency tables on $n \geq 4$ binary random variables 
with a specified collection of margins, denoted by $\Delta_n$, such that the gap between 
the linear programming approximation 
of the cell bounds and the true integer programming cell bound for one of these margins is $2^{n-3}-1$. 
Thus, we are not entitled to {\em always} assume that the linear approximation of the bounds 
on cell entries is a faithful approximation to the true bounds on these cell entries. 

All that said, perhaps the reason for \cite{CDKRM} \& \cite{DobFie} thinking that the linear 
approximation was reliable for bounding cell entries was that, in their practical 
experience, margins with large gaps were never encountered. 
Instead, perhaps the right claim for the practitioners to make is that the 
linear relaxations of the integer programs associated to disclosure limitation are 
{\em almost always} good approximations for bounding cells in contingency tables. 
In other words, while it is the case that Sullivant's specified margins do indeed 
provide examples of the linear approximation being poor, these margins may be rarely 
encountered in practice. We argue here that this is the case for Sullivant's large gaps 
and thus reopening the question of {\em whether linear programming is an effective heuristic 
for detecting disclosures when releasing margins of multi-way tables}:

\begin{theorem} \label{thm:main}
The margins for Sullivant's $(2^{n-3}-1)$-gaps on the models $\Delta^n$ are rare. 
\end{theorem}

We will make the notion of {\em rare} precise through {\em standard pairs}. In the sections that 
follow we will first review and fortify Sullivant's construction. Next, after defining standard 
pairs and arguing that they are the right tool to use for measuring {\em rarity}, we will prove 
Theorem~\ref{thm:main} using a series of propositions regarding standard pairs specific to the 
strengthened Sullivant construction. We will see that the strengthened construction is 
not just made for its own sake but is crucial to proving Theorem~\ref{thm:main}. 
Finally, the detection of the $(2^{n-3}-1)$-gaps are made possible through Gr{\" o}bner basis theory and this 
has been used to provide other examples of margins with large gaps \cite[Corollary 4.3]{HosStu}. 
We will close with computational results showing that these other instances of large gaps are 
also rare in the same sense of Theorem~\ref{thm:main}.

\section{An Alternative Construction Of Sullivant's Large Gaps}

Let $A$ be a fixed matrix with $N$ columns and ${\bf c}$ a real cost vector with $N$ entries. 
Then for every fixed ${\bf b}$ that's a non-negative integral combination of the columns of $A$ 
(i.e. ${\bf b} \in \N A$) we have the {\em integer program} 
$
\textup{IP}_{A,{\bf c}} ({\bf b}) := 
\textup{min} \{ {\bf c} \cdot {\bf u} \, : \, 
A{\bf u} = {\bf b}, \, {\bf u} \in \N^{N} \,  \}.
$ 
The {\em linear relaxation} $\textup{LP}_{A,{\bf c}} ({\bf b})$ of $\textup{IP}_{A,{\bf c}} ({\bf b})$ 
is simply the same problem with the constraint ${\bf u} \in \N^{N}$ replaced by ${\bf u} \geq {\bf 0}$ 
and ${\bf u}$ real. 
For each fixed ${\bf b} \in \N A$ we have the quantity 
$
\textup{gap}_{A,{\bf c}}({\bf b})
:= \, 
\textup{optimal value of} \,\textup{IP}_{A,{\bf c}} ({\bf b}) 
\, - \, \textup{optimal value of} \, \textup{LP}_{A,{\bf c}} ({\bf b}) 
$ 
and the maximum of all these is the {\em integer programming gap} \cite{HosStu} 
$
\textup{gap}_{\bf c}(A) \, = \, 
\textup{max}_{{\bf b} \in \N A} \{ \textup{gap}_{A,{\bf c}}({\bf b}) \}
$

We will be especially interested in scenarios where ${\bf c} := (1,{\bf 0}) (= {\bf e}_1 \in \R^N)$ 
but even in this case \cite[\S 4]{HosStu} computing the gap precisely can be challenging. 
However, using Gr{\" o}bner basis theory, we have the following proposition which is a quick 
way to get a lower bound on the integer programming gap. 
We repeat its proof here as we need its content for later results. 
From here on we replace $\textup{gap}_{{\bf e}_1}(A)$ by simply $\textup{gap}_{-}(A)$.

\begin{proposition} \cite[Corollary 4.3]{HosStu} \label{pro:degreeGap}
Let $A$ be an integer matrix, let ${\bf e}_1$ be a cost vector and let $\succ$ be any term order. 
Suppose ${\bf g} := {\bf u} - {\bf v}$ is a reduced Gr{\" o}bner 
basis element of $A$ (with respect to the weight 
order induced by ${\bf e}_1$ and $\succ$) with ${\bf u}_1 = \alpha \geq 2$ 
Then $\textup{gap}_{-}(A) \geq \alpha - 1 $.
\end{proposition}

\begin{proof} (this proof due to Seth Sullivant)
By our choice of ${\bf g}$, ${\bf u}$ is a non-optimal solution for the 
integer program with constraint vector $A{\bf u}$. 
Now consider the vector ${\bf u} - {\bf e}_1$ where ${\bf e}_1$ 
is the first standard unit vector. By the reducedness of ${\bf g}$, 
${\bf u} - {\bf e}_1$ must be an optimal solution for the integer program
$
\textup{min} \{ w_1  \, : \,   
A{\bf w}  = A ({\bf u} - {\bf e}_1), \, 
{\bf w} \, \textup{integral} \, \}
$. 
On the other hand, the vector 
$
{\bf w}^* := ({\bf u} - {\bf e}_1) - 
\frac{(\alpha - 1)}{\alpha}({\bf u} - {\bf v})
$
is a solution to the linear relaxation of the integer program 
with cost equal to zero, and so it must be an optimal solution 
to the linear relaxation of the above program.  Thus, this gives 
an instance showing the gap is greater than or equal to $\alpha - 1$. 
\end{proof}

We can rephrase Proposition~\ref{pro:degreeGap} as saying that 
$\textup{gap}_{A,{\bf e}_1}({\bf b}) = \alpha - 1$ where 
${\bf b} =  A({\bf u} - {\bf e}_1)$. Note too, by the same argument, 
${\bf u} - (\alpha - \beta){\bf e}_1$ is an 
optimal integer solution for the integer program whose linear relaxation has 
an optimal solution of 
$({\bf u} - (\alpha - \beta){\bf e}_1) - \frac{\beta}{\alpha}({\bf u} - {\bf v})$. 

\begin{corollary} \label{cor:degreeGap}
With the hypothesis of Proposition~\ref{pro:degreeGap}, for every integer 
$1 \leq \beta \leq \alpha - 1$, there exists a ${\bf b}$ such that 
$\textup{gap}_{A,{\bf e}_1}({\bf b}) = \beta$.
\end{corollary}

We will be interested in scenarios like the figure below. The problem of bounding cell entries 
is precisely that of 
placing lower \& upper bounds on, without loss of generality, the entry $u_{000}$ given that 
all the $u_{\bf i}$'s are non-negative and integral and that they sum in a manner described 
by the margins below. The linear relaxation of this problem is approximating the bounds by 
permitting the $u_{\bf i}$'s to be real valued and bounding the entry $u_{000}$ accordingly. 
We will focus on the discrepancy between the true lower bound and its linear approximation. 
In what follows we will review how the problem of finding the minimum value of the 
$u_{\bf 0}$ cell entry given the $\Set$ margins ${\bf b}$ is equivalent to solving 
an integer program $\textup{IP}_{A(\Set),(1,{\bf 0})}({\bf b})$ with the linear 
approximation being the linear relaxation of this program. The largest possible discrepancy 
for $\Set$ is precisely the integer programming $\textup{gap}_{-}(\Set)$. 

\noindent \parbox{2.8in}
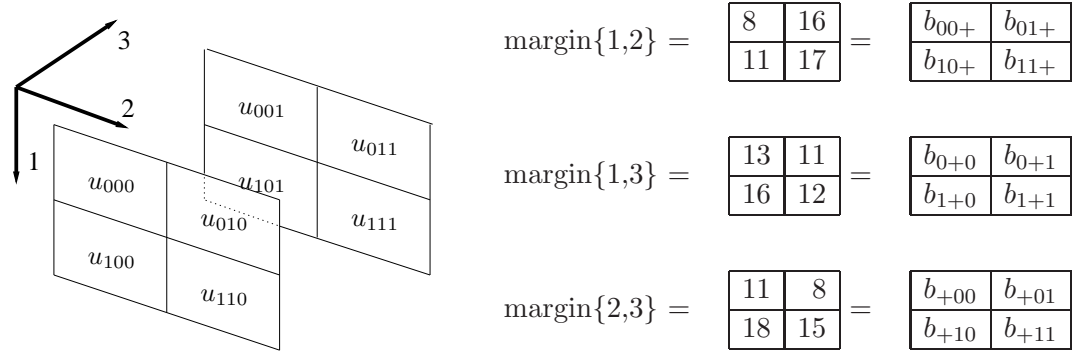
{\begin{picture}(0,0)%
\includegraphics{3by4by2Entry.pstex}%
\end{picture}%
\setlength{\unitlength}{2072sp}%
\begingroup\makeatletter\ifx\SetFigFontNFSS\undefined%
\gdef\SetFigFontNFSS#1#2#3#4#5{%
  \reset@font\fontsize{#1}{#2pt}%
  \fontfamily{#3}\fontseries{#4}\fontshape{#5}%
  \selectfont}%
\fi\endgroup%
\begin{picture}(5067,4016)(2627,-4573)
\put(3556,-2626){\makebox(0,0)[lb]{\smash{{\SetFigFontNFSS{10}{12.0}{\rmdefault}{\mddefault}{\updefault}{\color[rgb]{0,0,0}$u_{000}$}%
}}}}
\put(3556,-3526){\makebox(0,0)[lb]{\smash{{\SetFigFontNFSS{10}{12.0}{\rmdefault}{\mddefault}{\updefault}{\color[rgb]{0,0,0}$u_{100}$}%
}}}}
\put(5356,-2626){\makebox(0,0)[lb]{\smash{{\SetFigFontNFSS{10}{12.0}{\rmdefault}{\mddefault}{\updefault}{\color[rgb]{0,0,0}$u_{101}$}%
}}}}
\put(5356,-1726){\makebox(0,0)[lb]{\smash{{\SetFigFontNFSS{10}{12.0}{\rmdefault}{\mddefault}{\updefault}{\color[rgb]{0,0,0}$u_{001}$}%
}}}}
\put(6706,-2176){\makebox(0,0)[lb]{\smash{{\SetFigFontNFSS{10}{12.0}{\rmdefault}{\mddefault}{\updefault}{\color[rgb]{0,0,0}$u_{011}$}%
}}}}
\put(6706,-3076){\makebox(0,0)[lb]{\smash{{\SetFigFontNFSS{10}{12.0}{\rmdefault}{\mddefault}{\updefault}{\color[rgb]{0,0,0}$u_{111}$}%
}}}}
\put(4906,-3076){\makebox(0,0)[lb]{\smash{{\SetFigFontNFSS{10}{12.0}{\rmdefault}{\mddefault}{\updefault}{\color[rgb]{0,0,0}$u_{010}$}%
}}}}
\put(4906,-3976){\makebox(0,0)[lb]{\smash{{\SetFigFontNFSS{10}{12.0}{\rmdefault}{\mddefault}{\updefault}{\color[rgb]{0,0,0}$u_{110}$}%
}}}}
\end{picture}%
}
\parbox{4in}{ 
\hspace{-.3in}
margin\{1,2\}  =  \, \, 
  \begin{tabular}{ | l | r | }
    \hline
    8 & 16 \\ \hline
    11 & 17 \\ \hline
  \end{tabular}
\vspace{.2cm}
 =  \, \, 
  \begin{tabular}{ | l | r | }
    \hline
    $b_{00+}$ & $b_{01+}$ \\ \hline
    $b_{10+}$ & $b_{11+}$ \\ \hline
  \end{tabular} 

\vspace{.5cm}

\hspace{-.3in}
margin\{1,3\}  =  \, \, 
  \begin{tabular}{ | l | r | }
    \hline
    13 & 11 \\ \hline
    16 & 12 \\ \hline
  \end{tabular}
\vspace{.2cm}
 =  \, \, 
  \begin{tabular}{ | l | r | }
    \hline
    $b_{0+0}$ & $b_{0+1}$ \\ \hline
    $b_{1+0}$ & $b_{1+1}$ \\ \hline
  \end{tabular} 

\vspace{.5cm}

\hspace{-.3in}
margin\{2,3\}  =  \, \, 
  \begin{tabular}{ | l | r | }
    \hline
    11 & 8 \\ \hline
    18 & 15 \\ \hline
  \end{tabular}
\vspace{.2cm}
 =  \, \, 
  \begin{tabular}{ | l | r | }
    \hline
    $b_{+00}$ & $b_{+01}$ \\ \hline
    $b_{+10}$ & $b_{+11}$ \\ \hline
  \end{tabular} 

}

These margins can be described formally as follows: A {\em hierarchical model} is a 
simplicial complex $\Set$ on a ground set $[n] = \{ 1,2,\ldots,n\}$ together with 
an integer vector ${\bf d} = (d_1, d_2, \ldots, d_n)$. The quantity $n$ will be the 
dimension of our multiway contingency table and the $d_i$ is the number of levels 
in the $i^{\textup{th}}$ direction of the table. Every facet $F$ of $\Set$ 
indicates a margin to be released and we can always construct a matrix $A(\Set)$ to describe 
these margins. From here on we will assume that our contingency tables are {\em binary} 
($d_1=d_2 = \cdots = d_n =2$). See \cite{HosSul} for further discussion. 

Using Proposition~\ref{pro:degreeGap}, Sullivant \cite[Theorem 8]{SulLargeGaps} showed that 
for every $n \geq 4$ there is a model $\Delta_n$ with margin ${\bf b}$ such that 
$\textup{gap}_{-}(\Delta_n) \geq 2^{n-3}-1$. Sullivant constructs a {\em Graver basis} 
element ${\bf \hat{f}}_n$ with ${\bf 0}$-th entry equal to $2^{n-3}$ for $A(\Delta_n)$ and, 
because of the algebraic interpretation of $\Delta_n$ could then claim that this Graver basis 
element is part of a reduced Gr{\" o}bner basis. 

In the remainder of this section we will show something a little stronger while simultaneously 
giving a slightly easier proof of \cite[Theorem 8]{SulLargeGaps}: the constructed element 
${\bf \hat{f}}_n$ is in fact a {\em circuit} (a minimal linear dependence) for $A(\Delta_n)$. 
While the proof is similar to that of \cite[Theorem 8]{SulLargeGaps} it has the 
added benefit of avoiding the difficult {\em primitive} condition needed for Graver basis 
elements and, at the same time, showing that Sullivant's construction is really a modified 
version of the well known {\em checkerboard vector}. More importantly, we will need the 
stronger circuit property in the next section where we prove Theorem~\ref{thm:main}. 
Let's now build up Sullivant's construction.

\begin{example}
Suppose we have a $2 \times 2 \times 2$ contingency table with released margins 
specified by the simplicial complex $B = \{ \{1,2 \},\, \{1,3 \},\, \{2,3 \} \}$ (see Figure). 
If we replace the cell entries in the 3-dimensional binary table as follows: 
$
{\bf u} = (
u_{000}, u_{001},u_{010},u_{011},u_{100},u_{101},u_{110},u_{111}
)
= 
(5,3,8,8,6,5,10,7)
$ 
we then get the values in the released margins as specified in the figure. 

Another way of saying this is as follows: 
The computation of these margins is equivalent to the computation of 
$A(B){\bf u} = (b_{00+},b_{01+},b_{10+},b_{11+},b_{0+0}, \ldots,b_{+11})$. 
where the columns of the $0/1$-matrix $A(B)$ have indices  ${\bf i}$ which are 
ordered lexicographically 
$(0,0,0)$, $(0,0,1)$, $(0,1,0)$, $(0,1,1)$, $(1,0,0)$, $(1,0,1)$, $(1,1,0)$, $(1,1,1)$. 
In turn, the 
columns of $A(B)$ are labelled ${\bf e}_{\bf i}$ in order. For each face $F$ of $B$ 
we have a row matrix with $2^{|F|}$ rows and $2^n = 2^3$ columns: one row for each 
${\bf k} \in \{0,1 \}^{|F|}$ with the ${\bf i}^{th}$ entry in that row being equal to 
$1$ (and $0$ otherwise) if and only if ${\bf i}$ restricted to $F$, ${\bf i}_F$ equals 
${\bf k}$. For example, the row matrix for $\{ 1,2 \}$ of $B$ is 
$$
A(B)|_{\{ 1,2 \}} = 
\left[
\begin{array}{cccccccc} 
1 & 1  & 0  & 0 & 0   & 0  & 0 & 0  \\
0 & 0  & 1  & 1 & 0   & 0  & 0 & 0  \\
0 & 0  & 0  & 0 & 1   & 1  & 0 & 0  \\
0 & 0  & 0  & 0 & 0   & 0  & 1 & 1
\end{array}
\right]
\begin{array}{c} 
{\bf i}_{\{ 1,2 \}} = (0,0) \\
{\bf i}_{\{ 1,2 \}}  = (0,1)\\
{\bf i}_{\{ 1,2 \}}  = (1,0)\\
{\bf i}_{\{ 1,2 \}}  = (1,1)
\end{array}
$$
This proof of $A(B)$ being the matrix that describes precisely the margins of $\Set$ 
can be seen in \cite[Eqns. (3) \& (4)]{HosSul} and holds for any (binary or otherwise) 
hierarchical model. Also, the computational package {\tt 4ti2} \cite{4ti2} can be used 
to compute the matrix of margins for any (binary or otherwise) model.
\end{example}

The binary hierarchical model in the previous example is simply the boundary of the 
$3$-simplex. The first step in Sullivant's construction is based on the model 
$B_n = \{ S \subsetneq [n-2]\}$. The following lemma is a well known folklore result but 
we prove it here for the sake of completeness:
\begin{lemma} \label{lem:checker}
The kernel of the matrix $A(B_n)$ is $1$-dimensional. Letting ${\bf e}_{\bf i}$ be the 
standard unit vector in index position ${\bf i}$, the unique basis element (up to scalar 
multiplication) for this kernel is the checkerboard vector 
$
\textup{\bf checker} = \sum_{ {\bf i}: {\bf 1} \cdot {\bf i} \, \textup{even}} {\bf e}_{\bf i}
\, - \, 
\sum_{ {\bf i}: {\bf 1} \cdot {\bf i} \, \textup{odd}} {\bf e}_{\bf i}
$
\end{lemma}
\begin{proof}
From \cite[Theorem 2.6]{HosSul} the dimension of the kernel of $A(B_n)$ is 
exactly the number of elements in $2^{[n-2]}$ that are not in $B_n$. There is 
only one such element, $[n-2]$ itself, hence the kernel of $A(B_n)$ has 
dimension $1$. Since every subset of the columns of size $2^{n-1}$ is a 
linearly independent set then every column of $A(B_n)$ must be used 
non-trivially in the unique dependence, say ${\bf w}$, of $A(B_n)$. 
i.e., $w_{\bf i} \neq 0$ for every index ${\bf i} \in \{ 0,1\}^{n-2}$. 
Let the last component of ${\bf w}$ be $w_{\bf 1} = 1$.

From \cite[Lemma 2.1]{HosSul} this dependence must equal zero on 
every facet of the model $B_n$ and all the facets of $B_n$ are of the form 
$[n-2] \backslash \{ j\}$ where $j \in [n-2]$. 
For each facet $S := [n-2] \backslash \{ j\}$, there is precisely only one other index 
that creates a column with a non-zero entry in its ${\bf k} = {\bf 1}$ row 
of facet $S$ and this index is ${\bf 1} - {\bf e}_j$. Furthermore, 
since each of these entries equals $1$, then 
$w_{{\bf 1} - {\bf e}_j}$ must equal $-1$ for every $j \in [n-2]$. In other words, 
the index vectors ${\bf i}$ with ${\bf 1} \cdot {\bf i} = n-3$ must take the value 
$-1$ in ${\bf w}$.
Repeating this argument on each of the ${\bf i}$ with ${\bf 1} \cdot {\bf i} = n-3$, 
we get that ${\bf w}$ must also take the values of $+1$ on each of the positions indexed 
by ${\bf i}$ with ${\bf 1} \cdot {\bf i} = n-4$. Repeating recursively we get the vector 
${\bf w} = \textup{\bf checker}$ (up to possible sign change) as claimed.
\end{proof}

The {\em checkerboard vector} is so called because of the alternating $+1/-1$'s depending 
on parity. Next consider for each $n \geq 4$ the model
$
\Gamma_n := B_n \, \cup \, \{ \{ n-1 \} \}.
$ It is not too difficult to see that its matrix of margins is 
$
A(\Gamma_n) =
\left[
\begin{array}{ccc} 
A(B_n) & \vline & A(B_n)  \\ \hline
{\bf 1} & \vline &  {\bf 0} \\
{\bf 0} & \vline &  {\bf 1}
\end{array}
\right]
$ 
where the columns are indexed by all the $({\bf i}|0)$'s first followed by the 
$({\bf i}|1)$'s. The bottom two row vectors come from the row matrix for the face 
$\{ n-1\}$ of $\Gamma_n$. Permitting an abuse of notation, we write $\Gamma^n$ 
in place of $A(\Gamma_n)$.

Next, let $\sigma$ be the following index subset 
of the columns of $\Gamma^n$
$$\sigma = \{({\bf 0} | 0),\, ({\bf 0} | 1)\} \cup 
\{ ({\bf i}| 0) : {\bf 1} \cdot {\bf i} \, \textup{odd} \} 
\cup 
\{ ({\bf i}| 1) : {\bf i} \neq {\bf 0}, {\bf 1} \cdot {\bf i} \, \textup{even} \}. 
$$ 
Then the submatrix of $\Gamma^n$ indexed in order by these columns is 
$
\Gamma^n_\sigma =
\left[
\begin{array}{cccc} 
{\bf e}_{\bf 0} & {\bf e}_{\bf 0} & \vline & A(B_n)_{\overline{\bf 0}}  \\ \hline
1 & 0  & \vline &  {\bf r} \\
0 & 1  & \vline &  {\bf 1} - {\bf r}
\end{array}
\right]
$ where ${\bf r}$ equals $1$ in its first $2^{n-3}$ entries and $0$ in the remaining 
$2^{n-3}-1$ entries, ${\bf e}_{\bf 0}$ is the first column of $A(B_n)$ and 
$A(B_n)_{\overline{\bf 0}}$ is simply $A(B_n)$ with its first column indexed by 
${\bf 0}$ removed.

\begin{lemma} \label{lem:adaptChecker}
For each $n \geq 4$ there is a unique relation on the columns of the matrix $\Gamma^n_\sigma$ given by 
$$
{\bf f}_n := {\bf u}_n - {\bf v}_n := 
2^{n-3} {\bf e}_{({\bf 0} | 0)} + 
\sum_{{\bf i}: {\bf i} \neq {\bf 0}, {\bf 1} \cdot {\bf i} \, \textup{even}} 
{{\bf e}_{({\bf i}| 1)}}
\, - \, 
(2^{n-3} - 1) {\bf e}_{({\bf 0}| 1)} - 
\sum_{{\bf i}: {\bf 1} \cdot {\bf i} \, \textup{odd}} 
{{\bf e}_{({\bf i}| 0)}}.
$$ 
\end{lemma}
\begin{proof}
We first observe that $\textup{rank}(\Gamma^n_\sigma) = \textup{rank}(A(B_n)) + 1$. 
In addition, $\Gamma^n_\sigma$ has only one column more than $A(B_n)$ and so the kernel of 
$\Gamma^n_\sigma$ has dimension equal to $1$. The unique relation ${\bf f}_n$ 
(up to scalar multiplication) must also respect those relations on $A(B_n)$ and so 
$\Gamma^n_\sigma$ must respect the checkerboard relation of Lemma~\ref{lem:checker}. 
But then there are also the last two rows of $\Gamma^n_\sigma$ to account for, 
forcing the coefficients of ${\bf e}_{({\bf 0} | 0)}$ to be $2^{n-3}$ and of 
${\bf e}_{({\bf 0} | 1)}$ to be $2^{n-3}-1$ as claimed.
\end{proof}
\begin{corollary} \label{cor:circuit}
The vector ${\bf f}_n$ is a circuit (i.e. a minimal linear dependence) for the matrix $\Gamma^n$.
\end{corollary}

To complete Sullivant's construction, the {\em logit model} of $\Gamma_n$ is 
given by 
$
\Delta_n := \textup{logit}(\Gamma_n) = 
\{ T \cup \{ n \} \, : \, T \in \Gamma_n \} \, \cup \, 
2^{[n-1]} 
$.
The matrix $\Delta^{n}$ for the model $\Delta_n$ is the {\em Lawrence lifting} 
of $\Gamma^{n}$ \cite{SanStu} and equals 
$
\Delta^n = \left[
\begin{array}{ccc}
\Gamma^n & \vline & 0 \\
0 & \vline & \Gamma^n \\
I & \vline & I
\end{array}
\right] 
$
where $I$ is the $2^{n-1}$ identity matrix. Since $\Delta^n$ is the Lawrence lifting 
of $\Gamma^n$ then \cite[Chapter 7]{StuGBCP} we have the following 
property: ${\bf g}$ is an (integer) vector in the kernel of $\Gamma^n$ with 
${\bf g} \in \R^{2^{n-1}}$ if and only if the lifted 
${\bf \hat{g}} := ({\bf g},-{\bf g}) \in \R^{2^n}$ is in the (integer) kernel of $\Delta^n$.
Consequently, if we have a vector ${\bf g} \in \R^{2^{n-1}}$ in the kernel of $\Gamma^n$ and if 
the support of ${\bf g}$ equals $\tau \subseteq [2^{n-1}]$ then we denote the support of 
${\bf \hat{g}}$ by $\hat{\tau} \subseteq [2^n]$ and we can also describe $\hat{\tau}$ 
as follows: if $\eta \in \{ 0,1 \}^{n-1}$ then 
$
\eta \in \tau \Longleftrightarrow \, \textup{both} \, (\eta|0), \, (\eta|1) \in \hat{\tau}.
$ 
Because of the isomorphism ${\bf g} \longleftrightarrow ({\bf g},-{\bf g})$ between the 
kernels of matrices for any model and its logit, we have the first part of the following 
corollary. The second part follows from the first.
\begin{corollary} \textup{(1)} The lifted ${\bf \hat{f}}_n$ is a circuit for $\Delta^n$. 
\textup{(2) (from \cite[Prop. 4.11 \& Thm. 7.1]{StuGBCP})} The lifted ${\bf \hat{f}}_n$ is a reduced 
Gr{\" o}bner basis element for $\Delta^n$.
\end{corollary}

The lifted ${\bf \hat{f}}_n$'s are precisely what Sullivant constructed and argued that 
these were reduced Gr{\" o}bner basis elements for $\Delta^n$ by showing that they were 
{\em Graver basis} elements, which is a weaker condition than (1) above but still 
sufficient to be able to apply condition (2). Consequently, applying 
Proposition~\ref{pro:degreeGap} to ${\bf \hat{f}}_n$ we get that 
$\textup{gap}_{-}(\Delta_n) \geq 2^{n-3}-1$ and that this gap is given by the margin 
${\bf b} := \Delta^n({\bf \hat{u}}_n - {\bf e}_{\bf 0})$. In addition, from this margin 
${\bf b}$, other margins that provide the same size gap are easy to construct. 
In the next section, via standard pairs, we show that for each $n \geq 4$ all these 
margins are very rare.

\section{Rare Encounters With Large Gaps Via Standard Pairs}

In order to examine the frequency with which the $(2^{n-3}-1)$-gap occurs 
we need the notion of {\em standard pairs} \cite{StuTruVog}.
Let $A$ be a fixed matrix with $N$ columns and ${\bf c}$ a cost vector 
with $N$ entries. If $\gamma \in \N^N$ and $\tau \subseteq [N]$, we denote 
by $(\gamma, \tau)$ the set of vectors 
$\{ \gamma +  \sum_{l \in \tau} n_l {\bf e}_l \, : \, \, n_l \in \N \}.$ 
We call $\gamma$ the {\em root} of the pair and $\tau$ the {\em free directions} 
of the pair. We say that the pair is {\em associated} for the family of integer programs 
$\textup{IP}_{A,{\bf c}} := \{ \textup{IP}_{A,{\bf c}}({\bf b}) \, : \, {\bf b} \in \N A\}$ 
if both 
(i) $\textup{supp}(\gamma) \cap \tau = \emptyset$ and 
(ii) every vector ${\bf p} \in (\gamma, \tau)$ 
is an optimal solution for $\textup{IP}_{A,{\bf c}}(A{\bf p})$.
Furthermore, if there does not exist another associated pair 
$(\gamma^\prime, \tau^\prime)$ with 
$(\gamma, \tau) \subsetneq  (\gamma^\prime, \tau^\prime)$ then we say that 
$(\gamma, \tau)$ is a {\em standard pair} for the family of integer programs. 

We will now prove Theorem~\ref{thm:main} relying on the verification of propositions that 
will follow. As in the previous section, $\hat{\sigma}$ denotes the support of ${\bf \hat{f}}_n$, 
the Gr{\" o}bner basis element of $\Delta^n$ which we showed was also a circuit of $\Delta^n$. 
As before, let the indices of the columns of $\Delta^n$ be denoted by $({\bf i}|l|l^\prime)$ 
where ${\bf i} \in \{ 0,1\}^{n-2}$ and $l, l^\prime \in \{ 0,1\}$. 
Finally, let $M(q)$ denote the set of margins 
$\{ {\bf b} : \Delta^n {\bf x} = {\bf b} \, \textup{and} \, {\bf 1}\cdot{\bf x} = q \}$. Note 
that $\Delta^n$ is {\em graded} i.e. the entries in each column of $\Delta^n$ sum to $n$ and 
so every margin ${\bf b}$ belongs to a unique $M(q)$. 

\vspace{.3cm}

\noindent {\bf Proof of Theorem~\ref{thm:main}}: 
From Proposition~\ref{pro:degreeGap}, Sullivant's $(2^{n-3} -1)$-gap was created by the 
margin ${\bf b} = \Delta^n({\bf \hat{u}}_n - {\bf e}_{({\bf 0}|0|0)})$. In the remainder of this 
section we will show the following: 

{\bf Claim:} {\em The margin ${\bf b} = \Delta^n({\bf \hat{u}}_n - {\bf e}_{({\bf 0}|0|0)})$ belongs to 
the image (under $\Delta^n$) of the standard pair 
$((2^{n-3} -1) \cdot {\bf e}_{({\bf 0}|0|0)}, \hat{\sigma} \backslash \{ ({\bf 0}|0|0) \})$ 
and to no other standard pair.}

Note too that most elements from the standard pair 
$( (2^{n-3} - 1) \cdot {\bf e}_{({\bf 0}|0|0)}, \hat{\sigma} \backslash \{ ({\bf 0}|0|0) \})$ 
are (coordinate-wise) greater than ${\bf \hat{u}}_n - {\bf e}_{({\bf 0} | 0 |0)}$ and it follows 
easily that these elements also create gaps of size $2^{n-3} -1$. We will call the non-negative 
integer image under $\Delta^n$ of this standard pair the {\em Sullivant $(2^{n-3} -1)$-margins}.

These margins are rare in the following sense. By the gradedness of $\Delta^n$ each $M(q)$ 
is contained in the lattice points of a slice of the cone $\textup{cone}(\Delta^n)$. This cone has 
dimension $2^{n-1} + 2^{n-2}$ \cite[Theorem 2.6]{HosSul} and so each margin slice $M(q)$ 
of this cone is a collection of lattice points in a polytope of dimension $2^{n-1} + 2^{n-2} -1$. 
On the other hand, the Sullivant $(2^{n-3} -1)$-margins all live in the shifted cone 
$
\Delta^n({\bf u}_n - {\bf e}_{({\bf 0}|0|0)}) \, + \, 
\textup{cone}(\Delta^n_{\hat{\sigma} \backslash \{ ({\bf 0}|0|0)\}})
$ 
which has dimension $|\hat{\sigma}| - 1 = 2(2^{n-2}+1) - 1 = 2^{n-1}+1$. Consequently 
for each $q$, the $(2^{n-3}-1)$-margins sit in a $2^{n-1}$-dimensional slice of the 
$(2^{n-1}+2^{n-2}-1)$-dimensional $M(q)$. Hence, the 
$(2^{n-3}-1)$-margins sit in a relatively very small slice of $M(q)$ for every $q$ 
and would thus, in a random uniform choice of margin from $M(q)$, be rarely encountered. 
 \hfill $\square$

Before proving the central claim of the above proof there are a number of things to 
note from the above analysis of the $(2^{n-3}-1)$-margins. 
A reasonable alternative approach to measuring the frequency of these $(2^{n-3}-1)$-margins 
would be to ignore the standard pair analysis and instead ask how frequently these 
margins occur asymptotically. i.e. among all the margins with large $1$-norms. In this 
case, regardless of the model $\Set$, most gaps are $0$. This makes reasonable sense 
when phrased as {\em the linear relaxation of an integer program has an integer solution 
for most right hand sides ${\bf b}$ when $|{\bf b}| \gg 0$}. A clean algebraic statement in terms of 
the Hilbert function of toric ideals can be seen in \cite[Prop. 12.16]{StuGBCP}. 
However, given that released margins with large norms will come from tables 
with large cell counts (which are harder to bound tightly and are consequently more secure), 
the asymptotic results are not relevant thus justifying the need for the analysis above.

Also note that since the cone is shifted significantly from the origin then 
the $(2^{n-3}-1)$-margins only start to appear in $M(q)$ slices for $q \geq 2^{n-3} -1$ 
and so on contingency tables with mostly small counts, instances of these large gaps 
will never be encountered. Contingency tables with small cell counts are common in 
practice (see, for example, \cite{CDKRM, FieSla}) which could give a further explanation 
as to why the large gaps are not encountered in practice. 

Finally, we noted that a {\em random uniform choice} of 
margin in $M(q)$ is highly unlikely to pick out a $(2^{n-3}-1)$-margin but there 
may be some prior distribution on the margins that is not uniform. A very reasonable 
assumption, based on sums being distributed normally, is that the margins that are 
most frequently encountered are those in the {\em centre} of $\textup{cone}(\Delta^n)$ 
but in our instance the $(2^{n-3}-1)$-margins appear in a shifted cone, shifted in a 
highly skewed fashion away from the centre of $\textup{cone}(\Delta^n)$ along 
one of the extreme rays of that cone. So in fact, the uniform assumption in 
the proof of Theorem~\ref{thm:main} may even be overly generous to the occurrence 
of the $(2^{n-3}-1)$-margins. 

We now turn our attention to proving the main claim in the proof of 
Theorem~\ref{thm:main}. We first need the following remark:
\begin{remark} \label{rem:circuit}
Given any matrix $A$ and a circuit ${\bf w}$ of $A$, every integer vector in 
the kernel of $A$ with support contained in the support of ${\bf w}$ must be 
an integer multiple of ${\bf w}$. In particular, if $w_1 \geq 1$ then 
${\bf w} - {\bf e}_1$ cannot be in the kernel of the matrix $A$.
\end{remark}

For our interests, where ${\bf c} := {\bf e}_{\bf 0}$ and 
$A = \Delta^n$, we can rename the family of integer programs as  
$\textup{IP}_{\Delta^n,-}$ and the optimality condition (ii) in the 
associated pairs as follows: 
\begin{center}
$(\textup{ii})_{-}$: $\nexists$ both $\{n_l \in \N: l \in \tau\}$ and ${\bf t} \in \N^{2^n}$ 
such that $t_{({\bf 0}|0|0)} < p_{({\bf 0}|0|0)}$ and 
$
\Delta^n {\bf p} := \Delta^n (\gamma +  \sum_{l \in \tau} n_l {\bf e}_l) = \Delta^n {\bf t}
$
\end{center}

\begin{proposition} \label{pro:assPair}
The pair $(k \cdot e({\bf 0}|0|0), \hat{\sigma} \backslash \{ ({\bf 0}|0|0) \})$ is an 
associated pair for $\textup{IP}_{\Delta^n, -}$ for all $1 \leq k \leq 2^{n-3}-1$.
\end{proposition}
\begin{proof}
Clearly each pair satisfies condition (i) above so all we need verify is the rephrased 
condition $(\textup{ii})_{-}$. We will first show that condition 
$(\textup{ii})_{-}$ holds for $k = 2^{n-3}-1$. 

Let ${\bf p} = (2^{n-3}-1) \cdot {\bf e}_{({\bf 0}|0|0)} + \sum_{l \in \hat{\sigma}} n_l {\bf e}_l$. 
If there were to exist a ${\bf t}$ such that $\Delta^n {\bf p} = \Delta^n {\bf t}$ with 
$t_{({\bf 0}|0|0)} < 2^{n-3}-1$ then ${\bf p} - {\bf t}$ would be an integer vector 
in the kernel of $\Delta^n$. 
Since both ${\bf p}$ and ${\bf t}$ are non-negative, we can assume that their respective supports 
are disjoint. Since ${\bf p}$ is the positive part of the kernel element then an index element 
$(\eta|0)$ (or $(\eta|1)$) is in the support of ${\bf p}$ if and only if $(\eta|1)$ 
(or $(\eta|0)$ respectively) is in the support of ${\bf t}$. Therefore, by the construction of 
${\bf \hat{f}}_n$ and since $\textup{supp}({\bf p}) \subseteq \hat{\sigma}$ we must have 
$\textup{supp}({\bf t}) \subseteq \hat{\sigma}$. 

But this cannot occur: if the set $\textup{supp}({\bf t})$ were contained in $\hat{\sigma}$ 
then $\textup{supp}({\bf p} - {\bf t})$ would also be contained in $\hat{\sigma}$ 
with the $({\bf 0}|0|0)$-entry of this integer vector being positive and 
less than $2^{n-3}-1$, which using the fact that ${\bf \hat{f}}_n$ forms a circuit, 
would contradict Remark~\ref{rem:circuit}. Hence $(\textup{ii})_{-}$ is satisfied and so 
$( (2^{n-3}-1) \cdot e({\bf 0}|0|0), \hat{\sigma} \backslash \{ ({\bf 0}|0|0) \})$ is an 
associated pair. 

For the other values of $k$, we know that for any family of integer programs, 
if we have a vector ${\bf p}$ that is an optimal solution and another non-negative integral 
vector ${\bf p}^\prime$ with ${\bf p}^\prime \leq {\bf p}$ then ${\bf p}^\prime$ is also an 
optimal solution for that family. This proves that we have an associated pair for the other 
values of $k$ too. 
\end{proof}

Note that nothing special was used here about the matrix $\Delta^n$, only that it was 
the matrix of margins for a logit model, so we have the following general result for 
identifying quick gaps for other logit models:
\begin{corollary}
If $\hat{\bf f}$ be a circuit for any model $\textup{logit}(\Set)$ with 
$\textup{f}_{\bf 0} = \alpha$ and $\hat{\sigma} = \textup{supp}(\hat{\bf f})$. 
Then $(k \cdot {\bf e}_{\bf 0}, \hat{\sigma} \backslash \{ {\bf 0} \})$ is an 
associated pair for $\textup{IP}_{\textup{logit}(\Set),-}$ for all $1 \leq k \leq \alpha - 1$.
\end{corollary}

The next proposition claims that each of the associated pairs from the previous proposition 
are in fact standard pairs. 

\begin{proposition} \label{pro:stdPair}
The pair $(k \cdot {\bf e}_{({\bf 0}|0|0)}, \hat{\sigma} \backslash \{ ({\bf 0}|0|0) \})$ is a 
standard pair for $\textup{IP}_{\Delta^n,-}$ for all $1 \leq k \leq 2^{n-3}-1$.
\end{proposition}
\begin{proof}
It will suffice the consider the case of $k=1$. 
By the previous proposition, we know that the pair is associated. 
Recall that we have a containment of associated pairs 
$(\gamma, \tau) \subsetneq  (\gamma^\prime, \tau^\prime)$ if and only if 
$\gamma^\prime \leq \gamma$ and 
$\textup{supp}(\gamma - \gamma^\prime) \cup \tau \subset \tau^\prime$. 
If $\gamma = {\bf e}_{({\bf 0}|0|0)}$ and 
$\tau = \hat{\sigma} \backslash \{ ({\bf 0}|0|0) \}$ then such a 
$\gamma^\prime$ would equal $1$ or ${\bf e}_{({\bf 0}|0|0)}$. 
If $\gamma^\prime = 1$ then $({\bf 0}|0|0) \in \tau^\prime$ 
and $(1, \hat{\sigma})$ would have to be an associated pair, 
which cannot be the case since the non-optimal solution 
$\hat{\bf u}_n$ is in this pair. 

The other alternative is that $\gamma^\prime = {\bf e}_{({\bf 0}|0|0)}$ 
and in this case we need to show there does not exist an $l \notin \hat{\sigma}$ 
for which the pair 
$({\bf e}_{({\bf 0}|0|0)}, \hat{\sigma} \backslash \{ ({\bf 0}|0|0)  \} \cup \{ l \} )$ 
is an associated pair. Such $l$'s are of one of the following forms: 
(a) $({\bf i}|1|0)$, (b) $({\bf i}|0|1)$, (c) $({\bf i}|0|0)$ or (d) $({\bf i}|1|1)$ 
where ${\bf 0} \neq {\bf i} \in \{ 0,1\}^{n-2}$ as before. Note that regardless of 
the value of ${\bf i} \neq {\bf 0}$, we always have the relation 
$\Delta^n{\bf w}^+ = \Delta^n{\bf w}^-$ where 
$
{\bf w}^+ := {\bf e}_{({\bf 0}|0|0)} + {\bf e}_{({\bf i}|1|0)} 
         + {\bf e}_{({\bf 0}|1|1)} + {\bf e}_{({\bf i}|0|1)} 
$ 
and 
$
{\bf w}^- := {\bf e}_{({\bf 0}|1|0)} + {\bf e}_{({\bf i}|0|0)}
         + {\bf e}_{({\bf 0}|0|1)} + {\bf e}_{({\bf i}|1|1)}.
$ 
\begin{enumerate}
 \item[(a)] By construction, $l = ({\bf i}|1|0)$ is equivalent to $({\bf i}|0|0)$ and 
$({\bf i}|0|1)$ both being elements of $\hat{\sigma}$. We claim that the pair 
$(({\bf 0}|0|0), \hat{\sigma} \backslash \{ ({\bf 0}|0|0) \} \cup \{ ({\bf i}|1|0)\})$ 
violates condition $(\textup{ii})_{-}$. 
To see this notice that the choices of ${\bf p} = {\bf w}^+$ and ${\bf t} = {\bf w}^-$ 
satisfy the following: ${\bf p}$ has support in $\hat{\sigma} \cup ({\bf i}|1|0)$, 
that $\Delta^n {\bf p} = \Delta^n {\bf t}$ and $0 = t_{({\bf 0}|0|0)} < p_{({\bf 0}|0|0)} = 1$ 
Hence, this is such a choice for ${\bf p}$ and ${\bf t}$ violating $(\textup{ii})_{-}$. 
\item[(b)] The exact same choice of ${\bf p}$ and ${\bf t}$ can be made in this 
case as for part (a). 
\item[(c)] Consider the integral vector ${\bf \hat{f}}_n - (2^{n-3}-1)({\bf w}^+ - {\bf w}^-)$. 
This vector is in the kernel of $\Delta^n$ with $({\bf 0}|0|0)$-entry equal to $1$ 
and with positive support wholly contained in $\hat{\sigma} \cup \{ l \}$. Letting 
${\bf p}$ and ${\bf t}$ be the positive part and negative parts respectively of 
${\bf f}_n - (2^{n-3}-1)({\bf w}^+ - {\bf w}^-)$ we have a violation of condition 
$(\textup{ii})_{-}$.
\item[(d)] The exact same choice of ${\bf p}$ and ${\bf t}$ can be made in this 
case as for part (c). 
\end{enumerate}
Thus we have shown the $k=1$ case. For $2 \leq k \leq 2^{n-3}-1$ simply replace 
${\bf p}$ and ${\bf t}$ by $k \cdot {\bf p}$ and $k \cdot {\bf t}$ respectively.
\end{proof}

In the last proposition we created a set of standard pairs that contained the 
optimal solutions ${\bf \hat{u}}_n - (2^{n-3} - k){\bf e}_{({\bf 0} | 0 |0)}$ 
for every $1 \leq k \leq 2^{n-3} -1$. We can now complete the proof of the central 
claim in Theorem~\ref{thm:main}:
\begin{proposition} \label{pro:unique}
The optimal solution ${\bf \hat{u}}_n - {\bf e}_{({\bf 0} | 0 |0)}$ 
is contained in precisely one standard pair, namely, 
$((2^{n-3} -1) \cdot {\bf e}_{({\bf 0}|0|0)}, \hat{\sigma} \backslash \{ ({\bf 0}|0|0) \})$
\end{proposition}
\begin{proof}
The analysis is very similar to that which was carried out in Proposition \ref{pro:stdPair}.
We need to show that for any $l \notin \hat{\sigma}$, there exists $n_l \in \N$ such that 
${\bf \hat{u}}_n - {\bf e}_{({\bf 0} | 0 |0)} + n_l {\bf e}_l$ is not optimal. 
Such $l$'s are of one of the following forms: 
(a) $({\bf i}|1|0)$, (b) $({\bf i}|0|1)$, (c) $({\bf i}|0|0)$ or (d) $({\bf i}|1|1)$ 
where ${\bf 0} \neq {\bf i} \in \{ 0,1\}^{n-2}$ as before. 
The kernel element ${\bf w}^+ - {\bf w}^-$ of $\Delta^n$ is as above.
\begin{enumerate}
 \item[(a)] By construction, $l = ({\bf i}|1|0) \notin \hat{\sigma}$ is equivalent 
to ${\bf 1} \cdot {\bf i}$ being odd. If 
${\bf \hat{u}}_n - {\bf e}_{({\bf 0} | 0 |0)} + n_l {\bf e}_l$ were optimal for every 
$n_l \in \N$ then every vector 
${\bf 0} \leq {\bf z} \leq {\bf \hat{u}}_n - {\bf e}_{({\bf 0} | 0 |0)} + n_l {\bf e}_l$ 
would also be optimal. But, since ${\bf 1} \cdot {\bf i}$ is odd, then ${\bf z} = {\bf w}^+$ 
is such a vector and we already know that this vector is not optimal. 
\item[(b)] The case of $l = ({\bf i}|0|1) \notin \hat{\sigma}$ 
(with ${\bf 1} \cdot {\bf i}$ even) can be argued as in part (a).
\item[(c)] The next case is $l = ({\bf i}|0|0) \notin \hat{\sigma}$ with 
${\bf 1} \cdot {\bf i}$ being even. Here the index vector $({\bf i}|0|0)$ 
is in $\textup{supp}({\bf w}^-)$ and $({\bf i}|0|1)$ is in $\textup{supp}({\bf w}^+)$. 
Similar to the proof of Proposition \ref{pro:stdPair}, part (c), 
let ${\bf p}$ and ${\bf t}$ be the positive part and negative parts respectively of 
${\bf \hat{f}}_n - ({\bf w}^+ - {\bf w}^-)$. In this case, 
${\bf p} \leq {\bf \hat{u}}_n - {\bf e}_{({\bf 0} | 0 |0)} + {\bf e}_l$ and so it would 
need to be optimal if the free direction $l$ were to be allowed. However, 
$0 = t_{({\bf 0} | 0 |0)} < p_{({\bf 0} | 0 |0)} = 2^{n-3}-1$ and so we cannot 
have $l = ({\bf i}|0|0)$ with ${\bf 1} \cdot {\bf i}$ being even as a free direction 
in a standard pair that contains ${\bf \hat{u}}_n - {\bf e}_{({\bf 0} | 0 |0)}$.
\item[(d)] The case of $l = ({\bf i}|1|1) \notin \hat{\sigma}$ with 
${\bf 1} \cdot {\bf i}$ being odd is the same as that made in part (c).
\end{enumerate}
\vspace{-.7cm}
\end{proof}

\section{Closing Remarks}

Rather than asserting that {\em linear 
programming is an effective heuristic for detecting disclosures when 
releasing margins of multi-way tables} our result reopens this 
possibility, proposing that indeed large gaps in small hierarchical models 
do exist but may only rarely be encountered in practice. We have not addressed 
what happens for the other large gaps from Corollary~\ref{cor:degreeGap} 
that occur in the model $\Delta_n$. Nor have we addressed the extent to which 
the rarity encountered here happens for other hierarchical models. We attempt 
to address this by briefly reporting on some computational results. 

From Corollary~\ref{cor:degreeGap} and the discussion preceding it there 
are standard pairs like those from Proposition~\ref{pro:stdPair} whose respective 
images contain the $k$-margins for $1 \leq k \leq 2^{n-3}-2$ respectively. Using 
{\tt Macaulay 2} \cite{M2} we were able to confirm for $n=4$ and $n=5$ that while 
we did not have the uniqueness property of Proposition~\ref{pro:unique} for 
these $k$-gaps it was the case that the standard pairs $(\hat{\gamma}, \hat{\tau})$ 
for all of these margins 
had $|\hat{\tau}| = |\hat{\sigma}|-1$. Hence, for $n=4,5$ the computational evidence 
suggests that each of these $k$-gaps were contained 
in a $2^{n-1}$-dimensional slice of the $(2^{n-1} + 2^{n-2} -1)$-dimensional $M(q)$'s. 
Furthermore, they were all highly skewed along the $({\bf 0}|0|0)$ ray in the same manner 
as discussed after the proof of Theorem~\ref{thm:main} for the $(2^{n-3}-1)$-margins, 
again making these $k$-margins unlikely to be encountered.

Other instances of models with large gaps constructed from Proposition~\ref{pro:degreeGap} 
can be found in \cite[Prop. 2.7]{DevSul}. The binary model there is the collection of all 
edges of the complete graph with $n \geq 4$ vertices and a Gr{\" o}bner basis element 
(with respect to $(1,{\bf 0})$) is found 
there that provides a lower bound on the gap that grows linearly in $n$. In this case too 
the computational evidence using Macaulay 2 indicates that all margins attained from 
Proposition~\ref{pro:degreeGap} occur rarely. In the course of this work the answer to the 
following question seemed to be ``yes'' and may be of independent interest for those 
interested in {\em Markov moves} (see for example \cite[Ch. 1]{DrtStuSul}):

\begin{question} \label{ques:markov}
If ${\bf g} := {\bf u} - {\bf v}$ is a Gr{\" o}bner basis element 
for the matrix of margins for the model $\Set$ then is it true that \textup{(1)} 
$|\textup{supp}({\bf g})| > \textup{rank}(A(\Set))$ implies that all entries of 
${\bf g}$ belong to $\{-1,0,+1\}$~? \textup{(2)} Given any gap arising from 
Proposition~\ref{pro:degreeGap} is it true that its standard pair is always 
of the form $(\gamma, \tau)$ where $\tau \subsetneq \textup{supp}({\bf g})$~?
\end{question}

Thomas Kahle computationally verified that (1) is true for all models recorded at \cite{MBDB}. 
Note that if both (1) and (2) are true then every gap attained from 
Proposition~\ref{pro:degreeGap} would be rare in the sense of Theorem~\ref{thm:main}. 

Finally, the gaps coming from Proposition~\ref{pro:degreeGap} are not the only 
way that gaps can arise. The gap can be computed precisely \cite{HosStu} by solving a 
collection of {\em group relaxations} \cite[Ch. 24]{SchBlue} coming from the 
collection of standard pairs for $\textup{IP}_{\Delta^n, -}$. Using Macaulay 2, in 
the case of $n=5$ there are $1280$ such standard pairs $(\gamma, \hat{\tau})$ that 
need to be considered and $1013$ produce a gap greater than $0$. But when checked 
computationally for the $n=5$ case each of the standard pairs that had gap greater than 
or equal to 1 were exactly those that had the number of free directions strictly less 
(\& considerably less) than the rank of $\Delta^5$. Similarly for $n=4$ and $n=5$ in 
the case of the model studied in \cite[Prop. 2.7]{DevSul}.

In conclusion the computations using Macaulay 2 suggest that the results of 
Section 3 may be the typical scenario, that the gaps provided from 
Proposition~\ref{pro:degreeGap} may always be rare and furthermore that other gaps 
greater than or equal to $1$ may be equally rare. Thus the computations 
lend further support to linear programming being an effective heuristic 
for detecting disclosures when releasing margins of multi-way tables.

\section*{Acknowledgements} Many thanks to Seth Sullivant for suggesting the main problem 
of this paper and for answering my questions in the initial stages of this work. 
Thanks too to Thomas Kahle for computationally verifying the validity of part (1) of 
Question~\ref{ques:markov} for the many models recorded at the Markov Bases Database. 
A final acknowledgement to the developers 
of Macaulay 2 and 4ti2 without whose efforts the initial computational explorations 
in this work would have been a great deal more difficult.

\end{document}